\documentclass{article}
\usepackage{amsthm}
\usepackage{amsmath}
\usepackage{amsfonts}
\usepackage[letterpaper,body={13cm,21cm}, mag=1000]{geometry}
\usepackage{amssymb}
\usepackage{secdot}
\usepackage{setspace}
\usepackage{titlesec}
\titleformat{\section}[runin]{\bfseries\filcenter}{\thesection}{1em}{}

\renewcommand\footnotemark{}
\renewcommand{\thesection}{\arabic{section}}
\title{\large \bf Equality of certain automorphism groups of finite $p$-groups}
\author{\small \bf Deepak Gumber\footnote{The research of first author is supported by University Grants Commission of India under the Research Award Scheme} and Hemant Kalra\\
\small \em School of Mathematics and Computer Applications\\
\small \em Thapar University, Patiala - 147 004,
India\\
}
 
\date{}
\DeclareMathOperator{\IA}{IA}
\DeclareMathOperator{\Inn}{Inn}

\DeclareMathOperator{\Aut}{Aut}
\DeclareMathOperator{\Cent}{Aut_c}

\DeclareMathOperator{\Hom}{Hom}
\newtheorem{thm}{Theorem}[section]
\newtheorem{lm}[thm]{Lemma}

\newtheorem{cor}[thm]{Corollary}

\begin{document}
\maketitle
\begin{abstract}
\noindent  Let $G$ be a finite $p$-group and let $\Aut(G)$ denote the full automorphism group of $G$. In the recent past, there has been interest in finding necessary and sufficient conditions on $G$ such that certain subgroups of $\Aut(G)$ are equal. We prove a technical lemma and, as a consequence, obtain some new results and short and alternate proofs of some known results of this type.    
\end{abstract}

\vspace{2ex}


\vspace{2ex}


\section{Introduction} Let $G$ be a finite non-abelian $p$-group and let $M_1,M_2,N_1,N_2$ be normal subgroups of $G$. For normal subgroups $X$ and $Y$ of $G$, let $\Aut^X(G)$ and $\Aut_Y(G)$ denote the subgroups of $\Aut(G)$ centralizing $G/X$ and $Y$ respectively. We denote the intersection $\Aut^X(G)\cap\Aut_Y(G)$ by $\Aut_Y^X(G)$. In the recent past, many results have been proved which give necessary and sufficient conditions on $G$ such that $\Aut_Y^X(G)=\Inn(G),Z(\Inn(G))$; or $\Aut_{N_1}^{M_1}(G)=\Aut_{N_2}^{M_2}(G)$ with particular choices of $M_i$ and $N_i$ (see e.g. [3-5, 7-13]). Quite recently, Azhdari and Malayeri \cite[Theorem B, Corollary C]{azhmal} have found conditions on certain $M_i$ and $N_i$ so that $\Aut_{N_1}^{M_1}(G)=\Aut_{N_2}^{M_2}(G)$. We prove a short technical lemma, Lemma 2.2, and as a consequence, obtain very short and easy proofs of these main results of Azhdari and Malayeri. Subsequently, we also obtain some new results of this type and alternate proofs of main results of Attar \cite[Theorem A]{att3}, Jafari \cite[Theorem]{jaf} and Rai \cite[Theorem B(1)]{rai}.

Notations are mostly standard. By $\Hom(G,A)$ we denote the group of all homomorphisms of $G$ into an abelian group $A$ and by $C_n$ we denote the cyclic group of order $n$. The rank, exponent and nilpotence class of $G$ are respectively denoted as $d(G)$, $\exp(G)$ and $cl(G)$. A non-abelian group $G$ that has no non-trivial abelian direct factor is said to be purely non-abelian. An automorphism $\alpha$ of $G$ is called a central automorphism if it centralizes $G/Z(G)$, or equivalently, $x^{-1}\alpha(x)\in Z(G)$ for all $x\in G$. By $\Aut_c(G)$ we denote the group of all central automorphisms of $G$, and by $C^*$ we denote the group of all those central automorphisms of $G$ which fix $Z(G)$ element-wise. An automorphism $\alpha$ of $G$ is called an IA-automorphism if it centralizes the abelianized group $G/G'$. The group of all IA-automorphisms is denoted as  $\IA(G)$ and the group of all those IA-automorphisms which fix $Z(G)$ element-wise is denoted as $\IA(G)^*$.

\section{Automorphism groups of $G$}
While proving the equality of different automorphism groups of $G$, the foremost tool has been to express the group $\Aut^{X}_{Y}(G)$ in the form $\Hom(A,B)$ for suitable subgroups or quotient groups $A$ and $B$ of $G$. The trick is very well-known since old days. Our next lemma is a little modification of arguments of Alperin \cite[Lemma 3]{alp} and Fournelle \cite[Section 2]{fou}.

\begin{lm}\label{ML}
Let $G$ be any group and $X$ be a central subgroup of $G$ contained in a normal subgroup $Y$ of $G$. Then $\Aut^{X}_{Y}(G)\simeq\Hom(G/Y,X)$.
\end{lm}

Let $X$ and $Y$ be two finite abelian $p$-groups and let  $X\simeq C_{p^{x_1}}\times C_{p^{x_2}}\times\ldots\times C_{p^{x_h}}$ and $Y\simeq C_{p^{y_1}}\times C_{p^{y_2}}\times\ldots\times C_{p^{y_k}}$ be the cyclic decompositions of $X$ and $Y$, where $x_i\ge x_{i+1}$ and $y_i\ge y_{i+1}$ are positive integers. If either $X$ is a subgroup or is a quotient group of $Y$, then $h\le k$ and $x_i\le y_i$ for $1\le i\le h$. Consider the situation when $d(X)=d(Y)$ and $X$ is a proper subgroup or a proper quotient group of $Y$. In these circumstances, $h=k$ and there certainly exists an $r,\;1\le r\le h$, such that $x_r<y_r$ and $x_j=y_j$ for $r+1\le j\le h$. For this unique fixed $r$, let $var(X,Y)=p^{x_r}$. In other words, $var(X,Y)$ denotes the order of the last cyclic factor of $X$ whose order is less than that of corresponding cyclic factor of $Y$.

\begin{lm}\label{MT} Let $A,B,C$ and $D$ be finite abelian $p$-groups with $B$ a subgroup of $C$ and $D$ a quotient group of $A$. Then
\begin{enumerate}
\item[$(i)$] $\Hom(A,B)= \Hom(A,C)$ 
if and only if either $B=C$ or $d(B)=d(C)$ and $\exp(A)\le var(B,C)$,
\item[$(ii)$] $|\Hom(D,B)|=|\Hom(A,B)|$
if and only if either $D=A$ or $d(D)=d(A)$ and  $\exp(B)\le var(D,A)$.
\end{enumerate}
\end{lm}
\begin{proof} We prove only $(i)$ as the proof is similar for $(ii)$. Let
\[
\begin{array}{rcl}
A &\simeq & C_{p^{\alpha_1}}\times C_{p^{\alpha_2}}\times\ldots
\times C_{p^{\alpha_l}},\\
 
B & \simeq & C_{p^{\beta_1}}\times C_{p^{\beta_2}}\times\ldots \times C_{p^{\beta_m}}, \;\mbox{and}\\ 
C &\simeq & C_{p^{\gamma_1}}\times C_{p^{\gamma_2}}\times\ldots
\times C_{p^{\gamma_n}}
\end{array}
\]
be the cyclic decompositions of $A$, $B$ and $C$, where $\alpha_i\geq\alpha_{i+1},\;\beta_i\geq\beta_{i+1},$ and $\gamma_i\geq\gamma_{i+1}$ are positive integers.
 First suppose that 
$\Hom(A,B)= \Hom(A,C)$
 and $B<C$. Then $$\displaystyle\prod_{i=1}^l\displaystyle\prod_{j=1}^mp^{\mathrm {min}\{\alpha_i,
\beta_j\}}=\displaystyle\prod_{i=1}^l\displaystyle\prod_{k=1}^np^{\mathrm {min}\{\alpha_i,\gamma_k\}}.$$
Since $m\le n$ and $\beta_j\le \gamma_j$ for each $1\le j\le m$,  $\mathrm{min}\lbrace\alpha_i,\beta_j\rbrace\le \mathrm{min}\lbrace\alpha_i,\gamma_j\rbrace$. If $m<n$, then  $|\Hom(A,B)|<|\Hom(A,C)|$, which is not so. Thus $m=n$ and $\mathrm{min}\lbrace\alpha_i,\beta_j\rbrace= \mathrm{min}\lbrace\alpha_i,\gamma_j\rbrace$ for all $i$ and $j$. Let $var(B,C)=p^{\beta_r}$, $1\le r\le m$. Observe that $\exp(A)\le var(B,C)$, because if $p^{\alpha_1}>p^{\beta_r}$, then $\beta_r=\mathrm{min}\lbrace\alpha_1,\beta_r\rbrace=\mathrm{min}\lbrace\alpha_1,\gamma_r\rbrace>\beta_r$, a contradiction.

Conversely, if $B=C$, then result holds trivially. We therefore suppose that $B<C$, $d(B)=d(C)=m$ and  $\exp(A)\le var(B,C)$. Then $\alpha_i\le \alpha_1\le \beta_r<\gamma_r$ for $1\le i\le l$, and  $\beta_j=\gamma_j$ for $r+1\le j\le m$. It follows that
 
$$|\Hom(A,B)|=\prod_{i=1}^l\prod_{j=1}^mp^{\mathrm {min}\{\alpha_i,
\beta_j\}}=p^{r(\alpha_1+\cdots+\alpha_l)}\prod_{i=1}^l\prod_{j=r+1}^mp^{\mathrm {min}\{\alpha_i,\beta_j\}}$$
and
$$|\Hom(A,C)|=\prod_{i=1}^l\prod_{j=1}^mp^{\mathrm {min}\{\alpha_i,
\gamma_j\}}=p^{r(\alpha_1+\cdots+\alpha_l)}\prod_{i=1}^l\prod_{j=r+1}^mp^{\mathrm {min}\{\alpha_i,
\beta_j\}}.$$
Thus $|\Hom(A,B)|=|\Hom(A,C)|$ and hence $\Hom(A,B)=\Hom(A,C),$ because $\Hom(A,B)$ is a subgroup of $\Hom(A,C)$.
\end{proof}

\begin{cor}[{\em cf.} {\cite[Theorem B]{azhmal}}]
Let $G$ be a finite non-abelian $p$-group. Let $M_1$, $M_2$, $N_1$ and $N_2$ be normal subgroups of $G$ such that $M_i\le Z(G)\cap N_i$ for $i=1,2$, $M_1\le M_2$ and $N_2\le N_1$. Then $\Aut_{N_1}^{M_1}(G)=\Aut_{N_2}^{M_2}(G)$ if and only if one of the following conditions holds:
\begin{enumerate}

\item[$(i)$] $M_1=M_2$ and either $G/G'N_1=G/G'N_2$ or $d(G/G'N_1)=d(G/G'N_2)$ and  $\exp(M_1)\le var(G/G'N_1,G/G'N_2)$.
\item[$(ii)$] $G/G'N_1=G/G'N_2$ and either $M_1=M_2$ or  $d(M_1)=d(M_2)$ and $\exp(G/G'N_1)\le var(M_1,M_2)$.
\end{enumerate}
\end{cor}
\begin{proof} It follows from Lemma~\ref{ML} that for $i=1$ and 2,  
$$\Aut_{N_i}^{M_i}(G)\simeq\Hom(G/N_i, M_i)\simeq\Hom(G/G'N_i, M_i).$$
First suppose that $\Aut_{N_1}^{M_1}(G)=\Aut_{N_2}^{M_2}(G)$. Then 
\begin{equation}
|\Hom(G/G'N_1, M_1)|=|\Hom(G/G'N_2, M_2)|,
\end{equation}
and therefore either $M_1=M_2$ or $G/G'N_1=G/G'N_2$ by \cite[Lemma D]{curmac}. Assume that $M_1=M_2$ and $G/G'N_1$ is a proper quotient group of $G/G'N_2$. Then proof of $(i)$ follows from Lemma~\ref{MT}$(ii)$ by taking $D=G/G'N_1, \;A=G/G'N_2$ and $B=M_1=M_2$. Next assume that $G/G'N_1=G/G'N_2$ and $M_1<M_2$. Then it follows from (1) that 
$$\Hom(G/G'N_1, M_1)=\Hom(G/G'N_2, M_2).$$
The proof of $(ii)$ now follows from Lemma 2.2$(i)$ by taking $A=G/G'N_1=G/G'N_2$, $B=M_1$ and $C=M_2$.

Conversely, first assume that condition $(ii)$ holds. Then, taking $A=G/G'N_1=G/G'N_2$, $B=M_1$ and $C=M_2$ in Lemma 2.2$(i)$, we get $\Hom(G/G'N_1, M_1)=\Hom(G/G'N_2, M_2)$. It follows that $|\Aut_{N_1}^{M_1}(G)|=|\Aut_{N_2}^{M_2}(G)|$ and hence $\Aut_{N_1}^{M_1}(G)=\Aut_{N_2}^{M_2}(G)$ because $\Aut_{N_1}^{M_1}(G)\le\Aut_{N_2}^{M_2}(G)$. Now assume that condition $(i)$ holds. Then, taking $D=G/G'N_1$, $A=G/G'N_2$ and $B=M_1=M_2$ in Lemma 2.2$(ii)$, we get $|\Hom(G/G'N_1, M_1)|=|\Hom(G/G'N_2, M_2)|$ and hence $\Aut_{N_1}^{M_1}(G)=\Aut_{N_2}^{M_2}(G)$.
\end{proof}

Taking $M_1=M, N_1=N$ and $M_2=N_2=Z(G)$ in the above corollary, we get the following.

\begin{cor}[{\em cf.} {\cite[Corollary C(i)]{azhmal}}]
Let $G$ be a finite non-abelian $p$-group.
Let $M$ and $N$ be normal subgroups of $G$ such that $M\le Z(G)\le N$. Then $\Aut{_{N}^{M}}(G)=C^*$ if and only if one of the following conditions hold:
\begin{enumerate}

\item[$(i)$] $M=Z(G)$ and either $G/G'N=G/G'Z(G)$ or $d(G/G'N)=d(G/G'Z(G))$ and  $\exp(M)\le var(G/G'N,G/G'Z(G))$.
\item[$(ii)$] $G/G'N=G/G'Z(G)$ and either $M=Z(G)$ or $d(M)=d(Z(G))$ and $\exp(G/G'N)\le var(M,Z(G))$.

\end{enumerate}
\end{cor}

\begin{cor} [{\em cf.} {\cite[Corollary C(ii)]{azhmal}}] Let $G$ be a finite non-abelian $p$-group. Let $M$ and $N$ be normal subgroups of $G$ such that $M\le Z(G)\le N$. Then $\Aut_{N}^{M}(G)=\Aut_c(G)$ if and only if one of the following conditions hold:
\begin{enumerate}

\item[$(i)$] $M=Z(G)$ and either $N\le G'$ or $d(G/G'N)=d(G/G')$ and $\exp(M)\le var(G/G'N,G/G')$.
\item[$(ii)$] $N\le G'$ and either $M=Z(G)$ or  $d(M)=d(Z(G))$ and $\exp(G/G')\le var(M,Z(G))$.
\end{enumerate}
\end{cor}
\begin{proof} Observe that $\Aut_{N}^{M}(G)\simeq \Hom(G/G'N,M)$. First suppose that $\Aut_{N}^{M}(G)=\Aut_c(G)$. Then $\Aut_c(G)=C^*$ because $M\le Z(G)\le N$. Thus $G$ is purely non-abelian by \cite[Lemma 2.4]{yad} and hence $|\Aut_c(G)|=|\Hom(G/G', Z(G))|$ by \cite[Theorem 1]{adnyen}. It follows that
$$|\Hom(G/G'N,M)|=|\Hom(G/G', Z(G))|,$$ 
  and therefore either $N\le G'$ or $M=Z(G)$ by \cite[Lemma D]{curmac}. The proof now follows as in Corollary 2.3.

Conversely, if condition $(ii)$ holds, then $Z(G)\le G'$; and if condition $(i)$ holds, then $d(G/G'N)=d(G/G')$ implies that $G'N$ and hence $Z(G)$ does not contain any generating element of $G$. In either of the cases, $G$ is purely non-abelian and hence $|\Aut_c(G)|=|\Hom(G/G', Z(G))|$ by \cite[Theorem 1]{adnyen}. The proof now follows as in Corollary 2.3.
 \end{proof}

Some particular cases of this corollary give the following results of Rai \cite{rai}, Attar\cite{att3} and Jafari \cite{jaf}.

\begin{cor} [{\cite[Theorem B(1)]{rai}}]
Let $G$ be a finite non-abelian $p$-group. Then
$\IA(G)^*=\Aut_c(G)$ if and only if
$G'=Z(G)$.
\end{cor}
\begin{proof}
It is easy to see that if $G'=Z(G)$, then $\IA(G)^*=\Aut_c(G)$.
Conversely, if $\IA(G)^*=\Aut_c(G)$, then $cl(G)=2$ and thus $\IA(G)^*=\Aut_{Z(G)}^{G'}(G)=\Aut_c(G).$ The result now follows by taking $M=G'$ and $N=Z(G)$ in Corollary 2.5.
\end{proof}

\begin{cor}[{\em cf.} {\cite[Theorem A]{att3}} and {\cite[Theorem]{jaf}}] Let $G$ be a finite non-abelian $p$-group. Then
$\Aut_c(G)=C^*$ if and only if either $Z(G)\le G'$ or $d(G/G'Z(G))=d(G/G')$ and $\exp(Z(G))\le var(G/G'Z(G),G/G')$.
\end{cor}
\begin{proof}
The proof follows by taking $M=N=Z(G)$ in Corollary 2.5.
\end{proof}

We next obtain some new results which are immediate consequences of Lemma 2.2.

\begin{cor}
Let $G$ be a finite $p$-group of nilpotence class $2$. Then
$IA(G)=\IA(G)^*$ if and only if either $G'=Z(G)$ or $d(G/Z(G))=d(G/G')$ and $\exp(G')\le var(G/Z(G),G/G')$.
\end{cor}
\begin{proof} It follows from Lemma 2.1 that $\IA(G)\simeq \Hom(G/G',G')$ and $\IA(G)^*\simeq\Hom(G/Z(G),G').$ The result now follows from Lemma~\ref{MT}$(ii)$ by taking $D=G/Z(G)$, $A=G/G'$ and $B=G'$.
\end{proof}

\begin{cor}
Let $G$ be a finite non-abelian $p$-group. Then
$\IA(G)^*=C^*$ if and only if either $G'=Z(G)$ or $G'<Z(G)$, $d(G')=d(Z(G))$ and $\exp(G/Z(G))\le var(G',Z(G))$.
\end{cor}
\begin{proof} Observe that if $\IA(G)^*=C^*$, then nilpotence class of $G$ is 2 and hence $\IA(G)^*\simeq\Hom(G/Z(G),G')$ and $C^*\simeq\Hom(G/Z(G),Z(G))$ by Lemma 2.1. The result now follows from Lemma~\ref{MT}$(i)$ by taking $A=G/Z(G)$, $B=G'$ and $C=Z(G)$.
\end{proof}

\begin{cor}
Let $G$ be a finite non-abelian $p$-group. Then
$\IA(G)=C^*$ if and only if either $G'=Z(G)$ or $G'<Z(G)$, $d(G')=d(Z(G))$, $d(G/Z(G))=d(G/G')$ and $\exp(G')=var(G/Z(G),G/G')=\exp(G/Z(G))= var(G',Z(G))$.
\end{cor}
\begin{proof} That the conditions are sufficient follows from Corollaries 2.8 and 2.9. Conversely suppose that $\IA(G)=C^*$. Then nilpotence class of $G$ is 2 and $\IA(G)=\IA(G)^*=C^*$. It follows from Corollaries 2.8 and 2.9 that either $G'=Z(G)$ or $G'<Z(G)$, $d(G')=d(Z(G))$, $d(G/Z(G))=d(G/G')$ and $\exp(G')\le var(G/Z(G),G/G')\le \exp(G/Z(G))\le var(G',Z(G))\le \exp(G')$. Thus $\exp(G')=var(G/Z(G),G/G')=\exp(G/Z(G))= var(G',Z(G))$ and hence the result.
\end{proof}

As another application of Lemma 2.2, we next find necessary and sufficient conditions on a finite non-abelian $p$-group $G$ such that $\IA(G)=\Aut_c(G)$. We start with the following lemma.

\begin{lm}
Let $G$ be a finite nilpotent group of class $2$ such that $d(G')=d(Z(G))$. Then $G$ is purely non-abelian.
\end{lm}
\begin{proof} On the contrary, suppose that $G=H\times K$, where $H$ is a non-trivial abelian and $K$ is a purely non-abelian subgroup of $G$. Then $Z(G)=H\times Z(K)$ and $G'=K'\le Z(K)$. It follows that $d(G')\le d(Z(K))<d(Z(G))$, because $H$ is non-trivial. This is a contradiction and hence $G$ is purely non-abelian. 
\end{proof}

\begin{thm}
Let $G$ be a finite non-abelian $p$-group. Then
$\IA(G)=\Aut_c(G)$ if and only if either $G'=Z(G)$ or $G'<Z(G)$, $d(G')=d(Z(G))$ and $\exp(G/G')\le var(G',Z(G))$.
\end{thm}
\begin{proof}
First suppose that $\IA(G)=\Aut_c(G)$. Then nilpotence class of $G$ is 2 and $\IA(G)\simeq \Hom(G/G',G')$. We prove that $G$ is purely non-abelian. Assume that $G=A\times B$, where $A$ is abelian and $B$ is purely non-abelian. For $1\ne f\in\Hom(B, A)$, define $f^*:G\rightarrow G$ by $f^*(ab)=abf(b)$, $a\in A,\;b\in B$. It is easy to see that $f^*$ is a central but not an IA-automorphism of $G$. This is against the assumption and thus $G$ is purely non-abelian. It now follows from \cite[Theorem 1]{adnyen} that $|\Cent(G)|=|\Hom(G/G',Z(G))|$. Thus $|\Hom(G/G',G')|=|\Hom(G/G',Z(G))|$ and hence $\Hom(G/G',G')=\Hom(G/G',Z(G))$, because $\Hom(G/G',G')$ is a subgroup of $\Hom(G/G',Z(G))$. The result now follows from Lemma~\ref{MT} $(i)$ on replacing $A$ by $G/G'$, $B$ by $G'$ and $C$ by $Z(G)$.

Conversely, if $G'=Z(G)$, then trivially $\IA(G)=\Aut_c(G)$. Therefore suppose that $G'<Z(G)$, $d(G')=d(Z(G))$ and $\exp(G/G')\le var(G',Z(G))$. By Lemma 2.11, $G$ is purely non-abelian. It now follows by \cite[Theorem 1]{adnyen}, Lemma~\ref{ML} and Lemma~\ref{MT}$(i)$ that
$|\IA(G)|=|\Hom(G/G',G')|=
|\Hom(G/G',Z(G))|=|\Aut_c(G)|$
and hence $\IA(G)=\Aut_c(G)$, because $\IA(G)\le\Aut_c(G)$.
\end{proof}

\

\end{document}